\documentclass[12pt]{article}
\usepackage{amsmath,amssymb,amsthm}
\usepackage[english]{babel}
\usepackage[cp1251]{inputenc}
\def\leq {\leqslant}
\def\le {\leqslant}
\def\ge {\geqslant}
\def\geq {\geqslant}
\def\@bibitem[#1]#2{\item\@biblabel{#1}.\if@filesw
{\def\protect##1{\string##1\space}\immediate\write
\@auxout{\string\bibcite{#2}{#1}}}\fi\ignorespaces\@showtag{#2}}

\theoremstyle{plain}
\newtheorem{theorem}{Theotem}[section]
\newtheorem{rem}{Remark}[section]
\newtheorem{lemma}{Lemma}[section]
\newtheorem{op}{Definition}[section]

\renewcommand{\theequation}%
{\arabic{section}.\arabic{equation}}
\begin{document}
\pagestyle{myheadings}
\thispagestyle{empty}

\begin{center}
{\sc G. Akishev}
\end{center}

\begin{center}
{\bf Mixed modulus of smoothness and best approximation in Lorentz spaces
}
\end{center}

\vspace*{0.2 cm}

\begin{quote}
\noindent{\bf Abstract. }
The article deals with the mixed modulus of smoothness of positive order and the best approximation by ''angle'' of functions from the Lorentz space $L_{p, \tau}(\mathbb{T}^{m})$. The properties of the mixed modulus of smoothness, the sharp version of the inverse theorem of approximation theory are proved.
\end{quote}



\vspace*{0.2 cm}

\section*{Introduction} 

Let $\mathbb{N}$, $\mathbb{Z}$, $\mathbb{R}$ be sets of natural, integer, and real numbers, respectively, and $\mathbb{Z}_{+} = \mathbb{N }\cup\{0\}$, and $\mathbb{R}^{m}$ be an $m$--dimensional Euclidean space of points $\overline{x}=(x_{1},\ldots,x_{m})$ with real coordinates;
$\mathbb{T}^{m}=[0, 2\pi)^{m}$ and $\mathbb{I}^{m}=[0, 1)^{m}$ be the $m $--dimensional cubes. Further, $\mathbb{Z}^{m}$ and $\mathbb{Z}_{+}^{m}$ are $m$--fold Cartesian product of the sets $\mathbb{Z}$ and $\mathbb{Z}_{+}$, respectively.

 \smallskip
We denote by $L_{p,\tau}(\mathbb{T}^{m})$
the Lorentz space of all real-valued Lebesgue measurable functions $f$ that have $2\pi$--period in each variable and for which the quantity
\begin{equation*}
\|f\|_{p,\tau} = \left(\int\limits_{0}^{1}(f^{*}(t))^{\tau}t^{\frac{\tau}{p}-1}dt
\right)^{\frac{1}{\tau}} , \,\, 1<
p<\infty, 1\leqslant \tau <\infty,
\end{equation*}
is finite, where $f^{*}(t)$ is a non-increasing rearrangement of the function $|f(2\pi\overline{x})|$, $\overline{x} \in \mathbb{I}^{m }$ (see \cite[Ch.~1, sec. 3, P.~213--216]{22}).

In the case $\tau=p$, the Lorentz space $L_{p,\tau}(\mathrm{T}^{m})$ coincides with the Lebesgue space $L_{p}(\mathbb{T}^{m})$ with the norm (see \cite[Ch.~1, Sec.~1.1]{13}
\begin{equation*}
	\|f\|_{p} = \Biggl[\int_{0}^{1}...\int_{0}^{1}|f(2\pi x_{1},\dots, 2\pi x_{m})
	|^{p}dx_{1}\dots dx_{m}\Biggr]^{\frac{1}{p}}, \,\, 1\leq p<\infty.
\end{equation*}

We denote by ${\mathring L}_{p, \tau}(\mathbb{T}^{m})$  the set of all functions $f\in
L_{p, \tau}(\mathbb{T}^{m})$
such that
\begin{equation*}
\int\limits_{0}^{1}f(2\pi \overline{x}) dx_{j}  =0,\;\;
j=1,...,m .
\end{equation*}

Below $a_{\overline{n}}(f)$ are the Fourier coefficients of a function $f\in {\mathring L}_{1}\left(\mathbb{T}^{m} \right)$ with respect  the system $\{e^{i\langle\overline{n}, 2\pi\overline{x}\rangle}\}_{\overline{n}\in\mathbb{Z}^{m}}$ and $\langle\overline{y},\overline{x}\rangle=\sum\limits_{j=1}^{m}y_{j} x_{j}$;
$$
	\delta_{\overline{s}}(f, 2\pi\overline{x})
	= \sum\limits_{\overline{n} \in \rho \left( \overline{s} \right)}a_{\overline{n} }(f) e^{i\langle\overline{n}, 2\pi\overline{x}\rangle },
$$
where
$$
	\rho (\overline{s})=\left\{\overline{k} =\left( k_{1}, \dots, k_{m} \right) \in \mathrm{Z}^{m}\colon [2^{s_{j} -1}] \leq \left| k_{j} \right|<2^{s_{j} } ,j=1, \dots, m\right\},
$$
and $[a]$ is the integer part of $a$, $\overline{s} = (s_{1}, \dots, s_{m}), s_{j} = 0, 1, 2,\ldots$.
Further, $\mathbb{Z}_{+}^{m}$ is the set of points with non-negative integer coordinates.
The value
\begin{equation*}
Y_{l_{1},\ldots,l_{m}}(f)_{p, \tau} = \inf_{T_{l_{j}}} \|f-\sum_{j=1}^{m}T_{l_{j}}\|_{p, \tau} \;\;,\;\; l_{j} = 0,1,2,...
\end{equation*}
is called the best approximation by ''angle'' of a function $f\in {\mathring L}_{p, \tau}(\mathbb{T}^{m})$ by trigonometric polynomials, where $T_{l_{j}} \in {\mathring L}_{p, \tau}(\mathbb{T}^{m})$ is a trigonometric polynomial of order $l_{j}$ with respect to the variable $x_{j}, \,\, j = 1,\ldots,m$ (in the case $\tau=p$, see  \cite{14}--\cite{17}).

We consider the standard basis $\{\overline{e}_{k}\}_{k=1}^{m}$ in $\mathbb{R}^{m}$.

{\bf Definition} (see, for example, \cite{10}, \cite{17} and the bibliography therein).
Let $\alpha\in (0, \infty)$. For a function $f \in L(\mathbb{T}^{m})$, the positive order difference $\alpha$ with respect to the variable $x_{k}$ with step $h \in \mathbb{R}$ is determined by the formula
\begin{equation*}
\Delta_{h}^{\alpha}f(\overline{x})=\sum_{\nu=0}^{\infty}(-1)^{\nu}
\left(\begin{smallmatrix}\alpha \\ \nu \end{smallmatrix}\right) f(\overline{x}+(\alpha - \nu)h\overline{e}_{k}),
\end{equation*}
where $\left(\begin{smallmatrix}\alpha \\ \nu \end{smallmatrix}\right)=1$ for $\nu=0$, $\left(\begin{smallmatrix}\alpha \\ \nu \end{smallmatrix}\right)=\alpha$ for $\nu=\alpha$ and $\left(\begin{smallmatrix}\alpha \\ \nu \end{smallmatrix}\right)=\frac{\alpha(\alpha -1)\cdot \ldots \cdot (\alpha-\nu+1)}{\nu !}$ for $\nu\geqslant\alpha$.

The mixed difference of positive orders $\alpha_{k}$ with respect to the variable $x_{k}$ of a function $f \in L(\mathbb{T}^{m})$ is defined by induction as
\begin{equation*}
\Delta_{h_{m}}^{\alpha_{m}}(\Delta_{h_{m-1}}^{\alpha_{m-1}}(...(\Delta_{h_{1}}^{\alpha_{1}}f(\overline{x}))...))=\sum_{\nu_{m}=0}^{\infty}...\sum_{\nu_{1}=0}^{\infty}(-1)^{\nu_{1}+...+\nu_{m}}\prod_{j=1}^{m}
\left(\begin{smallmatrix}\alpha_{j} \\ \nu_{j} \end{smallmatrix}\right) f(\overline{x}+\sum_{i=1}^{m}(\alpha_{i} - \nu_{i})h_{i}\overline{e}_{i})
\end{equation*}
and denoted by the symbol $\Delta_{\overline{h}}^{\overline\alpha}f(\overline{x})$.

{\bf Definition} The mixed modulus of smoothness of order $\overline\alpha$ of a function $f\in L_{p, \tau}(\mathbb{T}^{m})$ is defined by the formula (in the case $\tau=p$, see \cite{10}, \cite[Ch. 1, Se. 11]{6})
\begin{equation*}
\omega_{\overline\alpha}(f, \overline{t})_{p, \tau}:=\omega_{\alpha_{1},...\alpha_{m}}(f, t_{1},..., t_{m})_{p, \tau}=\sup_{|h_{1}|\leq t_{1},...,|h_{m}|\leq t_{m}}\|\Delta_{\overline h}^{\overline\alpha} (f)\|_{p, \tau}.
\end{equation*}

For functions of one several variables, the concept of the modulus of smoothness of natural order and its properties are well known (see, for example, the monographs \cite{13}, \cite[Ch. 1, Sec. 11]{6}, \cite{24}, \cite[p. 97]{25}). The modulus of smoothness of positive order of one variable function $f\in L_{p}(\mathbb{T})$ is defined in \cite{8}, \cite{23} and extended to symmetric spaces in \cite{1}, \cite{26} (see also the bibliography in \cite{21}). Sharp inequalities between moduli of smoothness in different metrics of the Lebesgue space are established in \cite{20}, \cite{11}.

In the case $\tau=p$, the mixed modulus of smoothness of positive order of a function $f\in L_{p}(\mathbb{T}^{m})$ was defined and its properties were studied in \cite{17}, \cite{10}. The generalized mixed modulus of smoothness of natural order of a function $f\in L_{p}(\mathbb{T}^{m})$ was defined and its properties were studied by K.V. Runovskii and N.V. Omelchenko \cite{19}.

The development of the problem of direct and inverse theorems of approximation by trigonometric polynomials of a function $f\in L_{p}(\mathbb{T})$, $0<p<\infty$ are described in detail in the article by V.I. Ivanov \cite{9} and in the monograph by M.F. Timan \cite{25}. These questions for approximation by ''angle'' are investigated in \cite{14}--\cite{17}. The direct and inverse theorems of approximation by ''angle'' of periodic functions in the Lorentz space for a mixed modulus of smoothness of natural order are proved in \cite{2}, \cite{3}.

In the proposed article, one equivalent relation for the mixed modulus of smoothness of positive order of a function $L_{p, \tau}(\mathbb{T}^{m})$, $1<p, \tau<\infty$ is established, moreover, an improved version of the inverse theorem of approximation by ''angle'' in this space and the accuracy of this theorem in the space of functions with lacunary Fourier series are proved.

By $C(p,q,r,y)$ we denote positive constants depending on
the parameters indicated in parentheses, generally speaking, different in different
formulas.
For positive values $A(y), B(y)$, $A(y) \asymp
B(y)$ means that
there are positive numbers $C_{1},\,C_{2}$ such that $C_{1}A(y) \leq B(y) \leq C_{2}A(y)$. For the sake of brevity, in the case $B\ge C_{1}A$ or $B\le C_{2}A$, we often write $B \gg A$ or $B \ll A$, respectively.
  \smallskip

\setcounter{equation}{0}
\setcounter{lemma}{0}
\setcounter{theorem}{0}

\section{Auxiliary statements}\label{sec1}  

\smallskip

First, we introduce additional notation and give auxiliary statements.
Denote by $e_{m}$ the set of indices $\{1, \dots, m\}$, its arbitrary subset by $e$ and the number of elements of $e$ by $|e|$.
Here $\overline{r} =(r_{1},\ldots,r_{m})$ is an element of an $m$-dimensional space with
non-negative coordinates, and $\overline{r}^{e} =(r_{1}^{e},\ldots,r_{m}^{e})$
is the vector with components $r_{j}^{e} = r_{j}$ for $j\in e$ and $r_{j}^{e}= 0$ for $j\notin e.$

Let $\overline{l} = (l_{1},\ldots,l_{m})$ be an element of an $m$--dimensional space with
positive integer coordinates and a nonempty set $e\subset e_{m}.$
We set 
 \begin{equation*}
G_{\overline{l}}(e) = \{ \overline{k} = (k_{1},\ldots,k_{m})\in \mathbb{Z}^{m} :
|k_{j}|\le l_{j}, j\in e \quad |k_{j}| > l_{j}, j\notin e \}.
\end{equation*}

We consider various partial sums of Fourier series:
 \begin{equation*}
S_{\overline{l}}(f, 2\pi\overline{x}) = S_{l_{1},\ldots,l_{m}}(f, 2\pi\overline{x}) =
\sum\limits_{|k_{1}|\le l_{1}}\ldots\sum\limits_{|k_{m}|\le l_{m} }
a_{\overline{k}} (f) e^{i\langle\overline{k} , 2\pi\overline{x}\rangle}
 \end{equation*}
is the partial sum with respect to all variables;
 \begin{equation*}
S_{l_{j}, \infty}(f, 2\pi\overline{x}) = \sum\limits_{k_{1}=-\infty }^{+\infty}\ldots \sum\limits_{k_{j-1}=-\infty}^{+\infty}
\sum\limits_{|k_{j}|\le l_{1}}\sum\limits_{k_{j+1}=-\infty }^{+\infty}\ldots\sum\limits_{k_{m}=-\infty}^{+\infty}
a_{\overline{k}}(f) e^{i\langle\overline{k} , 2\pi\overline{x}\rangle}
 \end{equation*}
is the partial sum with respect to $ x_{j}\in [0, 1)$. 

In the general case, 
\begin{equation*}
S_{\overline{l}^{e}, \infty}(f, 2\pi\overline{x}) =
\sum\limits_{\overline{k} \in \prod_{j \in e} [-l_{j}, l_{j}]\times \mathbb{R}^{m-|e|} }
a_{\overline{k}}(f) e^{i\langle\overline{k} , 2\pi\overline{x}\rangle}
 \end{equation*}
is the partial sum with respect to $x_{j}\in [0, 1)$ for $j\in e.$ 
For a given subset of $e\subset e_{m}$, we put 
\begin{equation*}
U_{\overline{l}}(f, 2\pi\overline{x}) = \sum\limits_{e\subset e_{m}, e \neq \emptyset} \;\; \sum\limits_{\overline{k} \in G_{\overline{l}}(e)} a_{\overline{k} } (f) e^{i\langle\overline{k} ,2\pi\overline{x}\rangle}.
 \end{equation*}
In particular, for $m=2$, we have (see, for example, \cite{6})
 \begin{equation*}
U_{l_{1},l_{2}}(f, 2\pi\overline{x}) = S_{l_{1} ,\infty }(f, 2\pi\overline{x}) + S_{\infty,l_{2}  }(f, 2\pi\overline{x}) - S_{l_{1},l_{2}  }(f, 2\pi\overline{x}).
  \end{equation*}
 
 For a function $f\in \mathring{L}(\mathbb{T}^{m})$ and a vector $\overline{\alpha} = (\alpha_{1}, \ldots , \alpha_{m})$ with non-negative coordinates, the fractional differentiation operator is defined by the formula (see \cite[Chap. 3, Sec. 15]{4}) 
\begin{equation*}
f^{(\overline{\alpha})}(\overline{x}):= f^{(\alpha_{1}, \ldots , \alpha_{m})}(\overline{x}) = \sum_{\overline{n}\in \mathring{\mathbb{Z}}^{m}} \prod_{j=1}^{m}(i n_{j})^{\alpha_{j}}a_{\overline{n}}(f) e^{i\langle\overline{n} , \overline{x}\rangle},
\end{equation*}
where 
 $\mathring{\mathbb{Z}}^{m} = \{\overline{n} \in \mathbb{Z}^{m}: \, \, \prod_{j=1}^{m}n_{j} \neq 0 \}$ and $(i n_{j})^{\alpha_{j}} = |n_{j}|^{\alpha_{j}}e^{i\frac{\pi}{2}\alpha_{j} sign n_{j}}$, $j=1,..., m$.
    
Below we present some properties of the mixed modulus of smoothness of a function, which are proved by well-known methods as in \cite{17}, \cite{10}, \cite{25}, \cite{2}.
\begin{lemma}\label{lem1 1}
Let 
 $1 < p < +\infty$, $1\leqslant \tau<\infty$, $\alpha_{j}\in (0, \infty)$ for $j=1, \ldots , m$ 
and $f, g \in L_{p, \tau}(\mathbb{T}^{m})$. Then 

1. $\omega_{\overline{\alpha}}(f, \overline\delta^{e})_{p, \tau}=\omega_{\overline{\alpha}}(f, \overline{0})_{p, \tau}=0$,

2. $\omega_{\overline{\alpha}}(f + g, \overline\delta)_{p, \tau}\ll\omega_{\overline{\alpha}}(f, \overline\delta)_{p, \tau} + \omega_{\overline{\alpha}}(g, \overline\delta)_{p, \tau}$;
 
3. $\omega_{\overline{\alpha}}(f, \overline\delta)_{p, \tau}\leqslant \omega_{\overline{\alpha}}(f, \overline{t})_{p, \tau} \, \, \text{for} \, \, 0\leqslant \delta_{j}<t_{j}, \overline{t}=(t_{1}, \ldots , t_{m})$; 

4. $\prod_{j=1}^{m}\delta_{j}^{-\alpha_{j}}\omega_{\overline{\alpha}}(f, \overline\delta)_{p, \tau}\leqslant \prod_{j=1}^{m}t_{j}^{-\alpha_{j}}\omega_{\overline{\alpha}}(f, \overline{t})_{p, \tau}$ for  $0<t_{j}\leqslant \delta_{j}\leqslant 1$,   $j=1, \ldots , m$; 

5. $\omega_{\overline{\alpha}}(f, \lambda_{1}\delta_{1}, ...,\lambda_{m}\delta_{m})_{p, \tau} \ll \prod_{j=1}^{m}\lambda_{j}^{\alpha_{j}}\omega_{\overline{\alpha}}(f, \delta_{1}, ...,\delta_{m})_{p, \tau}$,
for numbers 
$\lambda_{j} \geq 1$, $j=1,...,m$;

6. $\omega_{\overline{\beta}}(f, \overline\delta)_{p, \tau}\leqslant \omega_{\overline{\alpha}}(f, \overline\delta)_{p, \tau}$ for $0<\alpha_{j}<\beta_{j}$, $j=1,...,m$, $\overline\beta=(\beta_{1}, \ldots,\beta_{m})$. 
\end{lemma}

Let 
 $n_{j}\in \mathbb{N}$, $j=1,...,m$. We consider 
the trigonometric polynomial  
 \begin{equation}\label{eq1.1}
T_{\overline{n}}(2\pi\overline{x}) = {\sum\limits_{k_{1}=-n_{1}}^{n_{1}}}' \ldots {\sum\limits_{k_{m}=-n_{m}}^{n_{m}}}'c_{\overline{k}}e^{i\langle\overline{k} , 2\pi\overline{x}\rangle}=\sum\limits_{\overline{0}<|\overline{k}|\leqslant\overline{n}}c_{\overline{k}}e^{i\langle\overline{k} , 2\pi\overline{x}\rangle}, \, \,  \overline{x}\in \mathbb{I}^{m}, 
\end{equation} 
where 
 $|\overline{k}|=(|k_{1}|,...,|k_{m}|)$ and the notation $\overline{0}<|\overline{k}|\leqslant\overline{n}$ means that $0<|k_{j}|\leqslant |n_{j}|$ for $j=1, \dots, m$.

\begin{lemma}\label{lem1 2} 
Let $1 < p < +\infty$, $1\leqslant \tau<\infty$,
$\alpha_{j}\in (0, \infty)$ for $j=1, \ldots , m$. Then for the derivative $T_{\overline{n}}^{(\alpha_{1},...,\alpha_{m})}(2\pi\overline{x})$ 
of the trigonometric polynomial $T_{\overline{n}}(2\pi\overline{x})$ of the form \eqref{eq1.1} 
the following inequalities hold:
\begin{equation*}
 \omega_{\overline{\alpha}}(T_{\overline{n}}, \delta_{1}, ...,\delta_{m})_{p, \tau}\ll 
\prod_{j=1}^{m}\delta_{j}^{\alpha_{j}}\|T_{\overline{n}}^{(\alpha_{1},...,\alpha_{m})}\|_{p, \tau}.,
\end{equation*}
\begin{equation*}
\|T_{\overline{n}}^{(\alpha_{1},...,\alpha_{m})}\|_{p, \tau}\ll \prod_{j=1}^{m}n_{j}^{\alpha_{j}}\|\Delta_{\frac{\pi}{n_{m}}}^{\alpha_{m}}(...(\Delta_{\frac{\pi}{n_{1}}}^{\alpha_{1}}T_{\overline{n}})...)\|_{p, \tau}.
\end{equation*}
\end{lemma}
\begin{proof} The first inequality in the Lemma is proved as Theorem 4.1 in \cite{17}. 
Let us prove the second inequality. For a positive-order derivative of the trigonometric polynomial of one variable 
  \begin{equation*}
 T_{n}(x) = {\sum\limits_{k=-n}^{n}}'c_{k}e^{ikx}
 \end{equation*}
 it is known that (see, for example, \cite{23}, \cite{18})
   \begin{equation}\label{eq1.2}
 T_{n}^{(\alpha)}(x) =\sum_{l=-\infty}^{\infty}d_{l}\Delta_{h}^{\alpha}T_{n}\bigl(x+\frac{l\pi}{n}-\frac{\alpha h}{2}\bigr),
 \end{equation}
where  
$\sum_{l=-\infty}^{\infty}|d_{l}|\ll n^{\alpha}$, $\alpha > 0$. We use the equality \eqref{eq1.2} $m$ times for a
the polynomial of the form \eqref{eq1.1} we have
   \begin{equation}\label{eq1.3}
T_{\overline{n}}^{(\overline\alpha)}(\overline{x})=\sum_{\nu_{1}=-\infty}^{\infty}...\sum_{\nu_{m}=-\infty}^{\infty}\prod_{j=1}^{m}d_{\nu_{j}}\Delta_{\overline{h}}^{\overline\alpha}T_{\overline{n}}\Bigl(x_{1}+\frac{l_{1}\pi}{n_{1}}-\frac{\alpha_{1}h_{1}}{2},..., x_{m}+\frac{l_{m}\pi}{n_{m}}-\frac{\alpha_{m}h_{m}}{2}\Bigr),
\end{equation} 
where 
 $\sum_{l_{j}=-\infty}^{\infty}|d_{l_{j}}|\ll n_{j}^{\alpha_{j}}$, $\alpha_{j} > 0$.
Next, using equality \eqref{eq1.3}, by virtue of the triangle inequality and the invariance of the norm in the Lorentz space with respect to the shift (see for example \cite{7}, \cite{12}), we obtain 
   \begin{equation*}
  \|T_{\overline{n}}^{(\overline\alpha)}\|_{p, \tau}\ll \prod_{j=1}^{m}\sum_{\nu_{j}=-\infty}^{\infty}|d_{\nu_{j}}|\|\Delta_{\overline{h}}^{\overline\alpha}T_{\overline{n}}\|_{p, \tau} \ll \prod_{j=1}^{m}n_{j}^{\alpha_{j}}\|\Delta_{\overline{h}}^{\overline\alpha}T_{\overline{n}}\|_{p, \tau},
   \end{equation*}
for $h_{j}=\frac{\pi}{n_{j}}$,  $j=1,...,m$.
\end{proof}
\begin{rem}\label{rem1.1}
In the case $\tau = p$, Lemma 1.2 was previously proved in \cite[Theorem 5. 1]{17}.
\end{rem}
For $e\subset e_{m}$, by $T_{\overline{n}^{e}, \infty}(\overline{x})$ we denote the trigonometric polynomial of order at most $n_{j}\in \mathbb{N}$ in the variable $x_{j}$ for $j\in e$. In the case $e=e_{m}$, the polynomial $T_{\overline{n}^{e}, \infty}(\overline{x})$ is defined by the equality \eqref{eq1.1}.
\begin{lemma}\label{lem1 3}
Let 
 $1 < p < +\infty$, $1\leqslant \tau<\infty$,
$\alpha_{j}\in (0, \infty)$ for
 $j\in e$. 
 Then for the derivative 
of the trigonometric polynomial $T_{\overline{n}^{e}, \infty}(2\pi\overline{x})$ 
the inequality holds
  \begin{equation*}
 \|T_{\overline{n}^{e}, \infty}^{(\overline{\alpha}^{e})}\|_{p, \tau}\ll \prod_{j\in e}n_{j}^{\alpha_{j}}\|\Delta_{\overline{h}^{e}(\overline{n})}^{\overline{\alpha}^{e}}T_{\overline{n}^{e}, \infty}\|_{p, \tau} 
\end{equation*}
for 
 $\overline{h}^{e}(\overline{n})=(h_{1}^{e}(n_{1}),\ldots , h_{m}^{e}(n_{m}))$, $h_{j}^{e}(n_{j})=\frac{\pi}{n_{j}}$,  $j\in e$.
\end{lemma}
\proof
For each variable $x_{j}$ for $j\in e$ using equality \eqref{eq1.2} we obtain
   \begin{equation}\label{eq1.4}
T_{\overline{n}^{e}, \infty}^{(\overline{\alpha}^{e})}(\overline{x})=\sum_{\nu_{j}=-\infty, j\in e}^{\infty}\prod_{j\in e}d_{\nu_{j}}\Delta_{\overline{h}^{e}}^{\overline{\alpha}^{e}}T_{\overline{n}^{e}, \infty}\Bigl(\overline{x}+\sum_{j\in e}\Bigl(\frac{l_{j}\pi}{n_{j}}-\frac{\alpha_{j}h_{j}}{2}\Bigr)\overline{e}_{j}\Bigr),
\end{equation}
where 
 $\sum_{l_{j}=-\infty}^{\infty}|d_{l_{j}}|\ll n_{j}^{\alpha_{j}}$, $\alpha_{j} > 0$.
Next, using equality \eqref{eq1.4} according to the triangle inequality and the invariance of the norm in the Lorentz space with respect to the shift (see for example \cite{7}, \cite{12}), we obtain the statement of Lemma 1. 3. 
  \hfill $\Box$
 
\begin{rem}\label{rem1.2}
In the case $\tau = p$ and $m=2$, Lemma 1.3 was previously proved in \cite{17} and in the Lebesgue space with mixed norm in \cite{18}.
 \end{rem}

\begin{lemma}\label{lem1 4}
(Bernstein's inequality). 
Let $1 <p, \tau < +\infty,$ $\alpha_{j}\in \mathbb{Z}_{+}$ for $j=1, \ldots , m$. Then for a trigonometric polynomial $T_{\overline{n}}$, the following inequality is true
\begin{equation*}
\|T_{\overline{n}}^{(\alpha_{1},...,\alpha_{m})}\|_{p, \tau}\ll\prod_{j=1}^{m}(n_{j}+1)^{\alpha_{j}} \|T_{\overline{n}}\|_{p, \tau}
\end{equation*}
\end{lemma}
 
\begin{lemma}(see \cite[Lemma 3]{3} )\label{lem1 5}
Let 
 $1 < p < +\infty,$\,\,$1 < \tau < +\infty$ and $f\in {\mathring L}_{p, \tau}(\mathbb{T}^{m})$. Then 
 \begin{equation*}
\|f - U_{l_{1},...,l_{m}} (f)\|_{p, \tau}\ll Y_{l_{1},...,l_{m}}(f)_{p, \tau}.
\end{equation*}
\end{lemma}
\begin{lemma}(Direct theorem, see \cite[lemma 4]{3})
\label{lem1 6}. Let $\alpha_{j}> 0$ for $j=1,\ldots,m$. If
 $f\in \mathring{L}_{p, \tau}(\mathbb{T}^{m})$,  $1<p<+\infty,$
$1<\tau <+\infty$, then 
 \begin{equation*}
Y_{\overline{n}}(f)_{p, \tau}\ll \omega_{\overline{\alpha}}
\Bigl(f,\frac{1}{n_{1}+1},...,\frac{1}{n_{m}+1}\Bigr)_{p, \tau}.
 \end{equation*}
\end{lemma}

\begin{lemma} (inverse theorem, see \cite{3}) 
\label{lem1 7}. 
If $1<\tau <+\infty$, $\alpha_{j}\in \mathbb{N}$, then 
 \begin{equation*}
 \omega_{\bar\alpha}\bigl(f,\frac{1}{n_{1}+1},...,\frac{1}{n_{m}+1}\bigr)_{p, \tau}
\ll\prod_{j=1}^{m}n_{j}^{-\alpha_{j}}\sum\limits_{\nu_{1}=1}^{n_{1}+1} \ldots \sum\limits_{\nu_{m}=1}^{n_{m}+1} \prod_{j=1}^{m}\nu_{j}^{\alpha_{j} - 1}Y_{\overline{\nu}}(f)_{p, \tau}.
 \end{equation*}
\end{lemma}

\begin{rem}\label{rem1.3} In the case $\tau = p$, Lemmas 1.3--1.5 are proved in \cite{14}, \cite{17}, and Lemma 1.1 in \cite{10}.
\end{rem}

  \smallskip

\setcounter{equation}{0}
\setcounter{lemma}{0}
\setcounter{theorem}{0}

\section{Main results}\label{sec2} 

\begin{theorem}\label{th1.1} Let $\alpha_{j}> 0$ for $j=1, \ldots , m$ and $1< p<\infty$, $1< \tau < \infty$. Then a function $f\in {\mathring L}_{p, \tau}(\mathbb{T}^{m})$ satisfies the relation
 \begin{equation*}
 \omega_{\overline\alpha}\bigl(f,\frac{\pi}{n_{1}},...,\frac{\pi}{n_{m}}\bigr)_{p, \tau} \asymp 
\|f - U_{\overline{n}} (f)\|_{p, \tau} + \sum_{e\subset e_{m}, e\neq \emptyset} \prod_{j\in e}n_{j}^{-\alpha_{j}} \|S_{\overline{n}^{e}, \infty}^{(\overline{\alpha}^{e})}(f-S_{\infty, \overline{n}^{\hat{e}}}(f))\|_{p, \tau},
  \end{equation*}
where $\hat{e}$ is the complement of the set $e$. 
\end{theorem}
\proof
From Lemma 1.5 and Lemma 1.6 it follows that  
 \begin{equation}\label{eq2.1}
 \|f - U_{\overline{n}} (f)\|_{p, \tau}\ll  \omega_{\overline\alpha}\bigl(f,\frac{\pi}{n_{1}},...,\frac{\pi}{n_{m}}\bigr)_{p, \tau}
\end{equation}
for a function 
 $f\in {\mathring L}_{p, \tau}(\mathbb{T}^{m})$.
We evaluate the norm 
  $\|S_{\overline{n}^{e}, \infty}^{\overline{\alpha}^{e}}(f-S_{\infty, \overline{n}^{\hat{e}}}(f))\|_{p, \tau}$. Let us introduce the notation 
   $G(\overline{x})=f(\overline{x})-S_{\infty, \overline{n}^{\hat{e}}}(f, \overline{x}))$. Then applying Lemma 1.3 we obtain 
 \begin{equation}\label{eq2.2}
 \|S_{\overline{n}^{e}, \infty}^{(\overline{\alpha}^{e})}(f-S_{\infty, \overline{n}^{\hat{e}}}(f))\|_{p, \tau}= \|S_{\overline{n}^{e}, \infty}^{(\overline{\alpha}^{e})}(G)\|_{p, \tau}\ll \prod_{j\in e}n_{j}^{\alpha_{j}}\|\Delta_{\overline{h}^{e}(\overline{n})}^{\overline{\alpha}^{e}} S_{\overline{n}^{e}, \infty}(G)\|_{p, \tau}.  
\end{equation}
Taking into account the linearity of the difference of positive order and the boundedness of the Fourier sum operator in the Lorentz space (see \cite{4}) from \eqref{eq2.2} we get
 \begin{equation}\label{eq2.3}
\|S_{\overline{n}^{e}, \infty}^{(\overline{\alpha}^{e})}(f-S_{\infty, \overline{n}^{\hat{e}}}(f))\|_{p, \tau} \ll \prod_{j\in e}n_{j}^{\alpha_{j}}\|S_{\overline{n}^{e}, \infty}(\Delta_{\overline{h}^{e}(\overline{n})}^{\overline{\alpha}^{e}}G) \|_{p, \tau}\ll \prod_{j\in e}n_{j}^{\alpha_{j}}\|\Delta_{\overline{h}^{e}(\overline{n})}^{\overline{\alpha}^{e}}G) \|_{p, \tau}
\end{equation}
for a function 
 $f\in {\mathring L}_{p, \tau}(\mathbb{T}^{m})$, $1< p<\infty$, $1< \tau < \infty$.
Let us introduce the notation 
 $\Delta_{\overline{h}^{e}(\overline{n})}^{\overline{\alpha}^{e}}f(\overline{x})=F(\overline{x})$.  Then again, taking into account the linearity of the difference operator of positive order, we have
 \begin{equation}\label{eq2.4}
\|\Delta_{\overline{h}^{e}(\overline{n})}^{\overline{\alpha}^{e}}G) \|_{p, \tau}=\|\Delta_{\overline{h}^{e}(\overline{n})}^{\overline{\alpha}^{e}}f -\Delta_{\overline{h}^{e}(\overline{n})}^{\overline{\alpha}^{e}}S_{\infty, \overline{n}^{\hat{e}}}(f) \|_{p, \tau}=\|F-S_{\infty, \overline{n}^{\hat{e}}}(F) \|_{p, \tau}.
\end{equation}
Inequalities \eqref{eq2.3} and \eqref{eq2.4} imply that  
\begin{equation}\label{eq2.5}
\|S_{\overline{n}^{e}, \infty}^{(\overline{\alpha}^{e})}(f-S_{\infty, \overline{n}^{\hat{e}}}(f))\|_{p, \tau} \ll \prod_{j\in e}n_{j}^{\alpha_{j}}\|F-S_{\infty, \overline{n}^{\hat{e}}}(F) \|_{p, \tau}. 
\end{equation}
Since $S_{\overline{0}^{e}, \infty}(F, \overline{x})=S_{\overline{0}^{e}, \overline{n}^{\hat{e}}}(F, \overline{x})=0$, then
\begin{equation}\label{eq2.6}
\|F-S_{\infty, \overline{n}^{\hat{e}}}(F) \|_{p, \tau}=\|F-S_{\overline{0}^{e}, \infty}(F)-S_{\infty, \overline{n}^{\hat{e}}}(F) +S_{\overline{0}^{e}, \overline{n}^{\hat{e}}}(F)\|_{p, \tau}
\end{equation}
Now, using Lemma 1.5 and Lemma 1.6, as well as the property of the mixed modulus of smoothness from \eqref{eq2.6}, we obtain
\begin{equation*} 
\|F-S_{\infty, \overline{n}^{\hat{e}}}(F) \|_{p, \tau}  \ll  \omega_{\overline\alpha}\bigl(F, \overline{1}^{e},  \overline{h}^{e}(\overline{n}))_{p, \tau}  \ll \sup_{|h_{j}|\leqslant \frac{\pi}{n_{j}}, j\in \hat{e}}\|\Delta_{\overline{h}^{\hat{e}}(\overline{n})}^{\overline{\alpha}^{\hat{e}}}F\|_{p, \tau}
\end{equation*}
\begin{equation}\label{eq2.7}
 =\sup_{|h_{j}|\leqslant \frac{\pi}{n_{j}}, j\in \hat{e}}\|\Delta_{\overline{h}^{\hat{e}}(\overline{n})}^{\overline{\alpha}^{\hat{e}}}(\Delta_{\overline{h}^{e}(\overline{n})}^{\overline{\alpha}^{e}}f)\|_{p, \tau} \ll  \omega_{\overline\alpha}\bigl(f,\frac{\pi}{n_{1}},...,\frac{\pi}{n_{m}}\bigr)_{p, \tau}.
\end{equation}
Now from inequalities \eqref{eq2.5} and \eqref{eq2.7} it follows that
\begin{equation}\label{eq2.8}
\|S_{\overline{n}^{e}, \infty}^{(\overline{\alpha}^{e})}(f-S_{\infty, \overline{n}^{\hat{e}}}(f))\|_{p, \tau} \ll \prod_{j\in e}n_{j}^{\alpha_{j}}\omega_{\overline\alpha}\bigl(f,\frac{\pi}{n_{1}},...,\frac{\pi}{n_{m}}\bigr)_{p, \tau}
\end{equation}
for a proper subset $e \subset e_{m}$. 

Let $e=e_{m}$. Then by Lemma 1.2 and due to the boundedness of the Fourier sum operator in the Lorentz space \cite{4}, we have
\begin{equation*} 
\|S_{\overline{n}}^{(\overline{\alpha})}(f)\|_{p, \tau}\ll \prod_{j=1}^{m}n_{j}^{\alpha_{j}}\|\Delta_{\frac{\pi}{n_{1}},...,\frac{\pi}{n_{m}}}^{\overline{\alpha}}S_{\overline{n}}(f)\|_{p, \tau}
=C\prod_{j=1}^{m}n_{j}^{\alpha_{j}}\|S_{\overline{n}}(\Delta_{\frac{\pi}{n_{1}},...,\frac{\pi}{n_{m}}}^{\overline{\alpha}}f)\|_{p, \tau}
\end{equation*}
\begin{equation}\label{eq2.9}
\ll \prod_{j=1}^{m}n_{j}^{\alpha_{j}}\omega_{\overline\alpha}\bigl(f,\frac{\pi}{n_{1}},...,\frac{\pi}{n_{m}}\bigr)_{p, \tau}
\end{equation}
for a function 
 $f\in {\mathring L}_{p, \tau}(\mathbb{T}^{m})$, $1< p<\infty$, $1< \tau < \infty$.
Now from inequalities \eqref{eq2.1}, \eqref{eq2.8} and \eqref{eq2.9} it follows that
 \begin{equation}\label{eq2.10}
 \begin{gathered}
 \|f - U_{\overline{n}} (f)\|_{p, \tau} + \sum_{e\subset e_{m}, e\neq \emptyset} \prod_{j\in e}n_{j}^{-\alpha_{j}} \|S_{\overline{n}^{e}, \infty}^{(\overline{\alpha}^{e})}(f-S_{\infty, \overline{n}^{\hat{e}}}(f))\|_{p, \tau} + \prod_{j=1}^{m}n_{j}^{-\alpha_{j}}\|S_{\overline{n}}^{(\overline{\alpha})}(f)\|_{p, \tau}
\\
\ll \omega_{\overline\alpha}\bigl(f,\frac{\pi}{n_{1}},...,\frac{\pi}{n_{m}}\bigr)_{p, \tau} 
\end{gathered}
\end{equation}
for a function 
 $f\in {\mathring L}_{p, \tau}(\mathbb{T}^{m})$, $1< p<\infty$, $1< \tau < \infty$.

Let us prove the opposite inequality to \eqref{eq2.10}. The function $ U_{\overline{n}} (f, \overline{x})$ can be written in the following form 
\begin{equation}\label{eq2.11}
 U_{\overline{n}} (f, 2\pi\overline{x})=\sum_{e\subset e_{m}, e\neq \emptyset}S_{\overline{n}^{e}, \infty}(f-S_{\infty, \overline{n}^{\hat{e}}}(f), 2\pi\overline{x}) + S_{\overline{n}}(f, 2\pi\overline{x}).  
\end{equation}
Now, using this equality and the property of the mixed modulus of smoothness, we get
  \begin{equation}\label{eq2.12}
 \begin{gathered}
\omega_{\overline\alpha}\bigl(f,\frac{\pi}{n_{1}},...,\frac{\pi}{n_{m}}\bigr)_{p, \tau} \leqslant \omega_{\overline\alpha}\bigl(f-U_{\overline{n}} (f), \frac{\pi}{n_{1}},...,\frac{\pi}{n_{m}}\bigr)_{p, \tau} + \omega_{\overline\alpha}\bigl(U_{\overline{n}} (f), \frac{\pi}{n_{1}},...,\frac{\pi}{n_{m}}\bigr)_{p, \tau} 
\\
\ll \|f - U_{\overline{n}} (f)\|_{p, \tau} + \sum_{e\subset e_{m}, e\neq \emptyset}\omega_{\overline\alpha}\bigl(S_{\overline{n}^{e}, \infty}(f-S_{\infty, \overline{n}^{\hat{e}}}(f), \frac{\pi}{n_{1}},...,\frac{\pi}{n_{m}}\bigr)_{p, \tau} + \omega_{\overline\alpha}\bigl(S_{\overline{n}}(f), \frac{\pi}{n_{1}},...,\frac{\pi}{n_{m}}\bigr)_{p, \tau}
\\
\ll \|f - U_{\overline{n}} (f)\|_{p, \tau} + \sum_{e\subset e_{m}, e\neq \emptyset} \prod_{j\in e}n_{j}^{-\alpha_{j}} \|S_{\overline{n}^{e}, \infty}^{(\overline{\alpha}^{e})}(f-S_{\infty, \overline{n}^{\hat{e}}}(f))\|_{p, \tau} + \prod_{j=1}^{m}n_{j}^{-\alpha_{j}}\|S_{\overline{n}}^{(\overline{\alpha})}(f)\|_{p, \tau}  
\end{gathered}
\end{equation}
for a function 
 $f\in {\mathring L}_{p, \tau}(\mathbb{T}^{m})$, $1< p<\infty$, $1< \tau < \infty$.
\hfill $\Box$

\begin{theorem}\label{th1.2}
Let $\alpha_{j}> 0$ for $j=1, \ldots , m$, $1< p<\infty$ and $1< \tau \leqslant 2$ or $2<p<\infty$ and $2<\tau<\infty$, $\beta=\min\{2, \tau\}$. Then for a function $f\in {\mathring L}_{p, \tau}(\mathbb{T}^{m})$ the inequality holds
  \begin{equation*}
 \omega_{\overline\alpha}\bigl(f,\frac{\pi}{n_{1}},...,\frac{\pi}{n_{m}}\bigr)_{p, \tau} \ll
 \prod_{j=1}^{m}n_{j}^{-\alpha_{j}}\Biggl(\sum\limits_{\nu_{1}=1}^{n_{1}+1} \ldots \sum\limits_{\nu_{m}=1}^{n_{m}+1} \prod_{j=1}^{m}\nu_{j}^{\beta\alpha_{j} - 1}Y_{\overline{\nu}}^{\beta}(f)_{p, \tau}\Biggr)^{1/\beta}, \,\, n_{j}\in \mathbb{N}.
 \end{equation*}
 \end{theorem}

\proof
This theorem is stated in \cite[Theorem 4.2]{2} and proved for $m=2$. We prove the theorem for $m\geqslant 3$.
For $n_{j}\in \mathbb{N}$, choose a non-negative integer $k_{j}$ such that
$2^{k_{j}}\leq n_{j}< 2^{k_{j} + 1}$, $j=1, \ldots , m$. Then, according to the property of the mixed modulus of smoothness of a function, we have
\begin{equation*}
   I_{1}(f):=\omega_{\overline\alpha}\Bigl(f,\frac{\pi}{n_{1}},\ldots, \frac{\pi}{n_{m}}\Bigr)_{p, \tau}  \ll \omega_{\overline\alpha}\Bigl(f,\frac{1}{2^{k_{1}}},\ldots , \frac{1}{2^{k_{m}}}\Bigr)_{p, \tau}.
 \end{equation*}
Further, according to the properties of the mixed modulus of smoothness of a function, we obtain
\begin{equation}\label{eq2.13}
I_{1}^{\tau}(f)\ll  \omega_{\overline\alpha}^{\tau}\bigl(f-U_{2^{k_{1}},...,2^{k_{m}}} (f), \frac{1}{2^{k_{1}}},\ldots , \frac{1}{2^{k_{m}}}\bigr)_{p, \tau} + \omega_{\overline\alpha}^{\tau}\bigl(U_{2^{k_{1}},...,2^{k_{m}}} (f), \frac{1}{2^{k_{1}}},\ldots , \frac{1}{2^{k_{m}}}\bigr)_{p, \tau}
\end{equation}
Using equality \eqref{eq2.11} and the property of the mixed modulus of smoothness, we have (see \eqref{eq2.12})
\begin{multline}\label{eq2.14}
\omega_{\overline\alpha}\bigl(U_{2^{k_{1}},...,2^{k_{m}}} (f), \frac{1}{2^{k_{1}}},\ldots , \frac{1}{2^{k_{m}}}\bigr)_{p, \tau} \ll \sum_{\substack{e\subset e_{m}, \\
e\neq \emptyset}} \omega_{\overline\alpha}\bigl(S_{2^{\overline{k}^{e}, \infty}}(f-S_{\infty, 2^{\overline{k}^{\hat{e}}}}(f), \frac{1}{2^{k_{1}}},\ldots , \frac{1}{2^{k_{m}}}\bigr)_{p, \tau}
\\
 + \omega_{\overline\alpha}\bigl(S_{2^{k_{1}},...,2^{k_{m}}}(f), \frac{1}{2^{k_{1}}},\ldots , \frac{1}{2^{k_{m}}}\bigr)_{p, \tau}.
\end{multline}
From inequalities \eqref{eq2.13} and \eqref{eq2.14} it follows that 
\begin{multline}\label{eq2.15}
I_{1}^{\tau}(f)\ll  \omega_{\overline\alpha}^{\tau}\bigl(f-U_{\overline{n}} (f), \frac{1}{2^{k_{1}}},\ldots , \frac{1}{2^{k_{m}}}\bigr)_{p, \tau} + \omega_{\overline\alpha}^{\tau}\bigl(S_{2^{k_{1}},...,2^{k_{m}}}(f), \frac{1}{2^{k_{1}}},\ldots , \frac{1}{2^{k_{m}}}\bigr)_{p, \tau} 
\\
+ \sum_{\substack{e\subset e_{m}, \\ e\neq \emptyset}}\omega_{\overline\alpha}^{\tau}\bigl(S_{2^{\overline{k}^{e}, \infty}}(f-S_{\infty, 2^{\overline{k}^{\hat{e}}}}(f), \frac{1}{2^{k_{1}}},\ldots , \frac{1}{2^{k_{m}}}\bigr)_{p, \tau}.
\end{multline}
By the property of the mixed modulus of smoothness and Lemma 1.5 we have
 \begin{equation}\label{eq2.16}
\omega_{\overline\alpha}^{\tau}\bigl(f-U_{2^{k_{1}},...,2^{k_{m}}} (f), \frac{1}{2^{k_{1}}},\ldots , \frac{1}{2^{k_{m}}}\bigr)_{p, \tau}\ll \|f-U_{2^{k_{1}},...,2^{k_{m}}} (f)\|_{p, \tau}\ll Y_{2^{k_{1}},...,2^{k_{m}}}(f)_{p, \tau}
\end{equation}
for a function 
  $f\in {\mathring L}_{p, \tau}(\mathbb{T}^{m})$, $1< p<\infty$, $1< \tau < \infty$.
Further, using Lemma 1.2 and Theorem 2.1 \cite{2} we obtain 
 \begin{multline}\label{eq2.17}
\omega_{\overline\alpha}\bigl(S_{2^{k_{1}},...,2^{k_{m}}}(f), \frac{1}{2^{k_{1}}},\ldots , \frac{1}{2^{k_{m}}}\bigr)_{p, \tau}\ll \prod_{j=1}^{m}2^{-k_{j}\alpha_{j}}\|S_{2^{k_{1}},...,2^{k_{m}}}^{(\overline{\alpha})}(f)\|_{p, \tau}
\\
\ll\prod_{j=1}^{m}2^{-k_{j}\alpha_{j}} \Biggl\|\Biggl(\sum\limits_{\mu_{1}=0}^{k_{1}}...\sum\limits_{\mu_{m}=0}^{k_{m}}\prod_{j=1}^{m}2^{\mu_{j}\alpha_{j}2}|\delta_{\overline{\mu}}(f)|^{2}\Biggr)^{1/2} \Biggr\|_{p, \tau}
\end{multline}
for a function $f\in {\mathring L}_{p, \tau}(\mathbb{T}^{m})$, $1< p<\infty$, $1< \tau < \infty$.

Let $\beta = \min\{2, \tau\}=2$ and $2< p< \infty$. Then according to Lemma 1. 2 \cite{2} from \eqref{eq2.17} we obtain
\begin{multline}\label{eq2.18}
\omega_{\overline\alpha}\bigl(S_{2^{k_{1}},...,2^{k_{m}}}(f), \frac{1}{2^{k_{1}}},\ldots , \frac{1}{2^{k_{m}}}\bigr)_{p, \tau}  \ll \prod_{j=1}^{m}2^{-k_{j}\alpha_{j}}\Biggl\{\sum\limits_{\mu_{1}=0}^{k_{1}}...\sum\limits_{\mu_{m}=0}^{k_{m}} \prod_{j=1}^{m}2^{\mu_{j}\alpha_{j}2}\|\delta_{\overline{\mu}}(f)\|_{p, \tau}^{2} \Biggr\}^{1/2}
\\
 \ll \prod_{j=1}^{m}2^{-k_{j}\alpha_{j}}\Biggl\{\sum\limits_{\mu_{1}=0}^{k_{1}}...\sum\limits_{\mu_{m}=0}^{k_{m}}\prod_{j=1}^{m}2^{\mu_{j}\alpha_{j}2} Y_{[2^{\mu_{1}-1}],... [2^{\mu_{m}-1}]}^{2}(f)_{p, \tau} \Biggr\}^{1/2}
\end{multline}
for a function  
$f\in L_{p, \tau}(\mathbb{T}^{m})$, 
in the case $\beta = \min\{2, \tau\}=2$ and $2< p< \infty$.

If $\beta = \min\{2, \tau\}=\tau$ and $1< p< \infty$, then by Lemma 1.1 and Lemma 1.4 \cite{4} from \eqref{eq2.17} we obtain
\begin{multline}\label{eq2.19}
\omega_{\overline\alpha}\bigl(S_{2^{k_{1}},...,2^{k_{m}}}(f), \frac{1}{2^{k_{1}}},\ldots , \frac{1}{2^{k_{m}}}\bigr)_{p, \tau}   \ll \prod_{j=1}^{m}2^{-k_{j}\alpha_{j}}\Biggl\{\sum\limits_{\mu_{1}=0}^{k_{1}}...\sum\limits_{\mu_{m}=0}^{k_{m}}\prod_{j=1}^{m}2^{\mu_{j}\alpha_{j}\tau}\|\delta_{\overline{\mu}}(f)\|_{p, \tau}^{\tau}\Biggr\}^{1/\tau}
\\
\ll \prod_{j=1}^{m}2^{-k_{j}\alpha_{j}}\Biggl\{\sum\limits_{\mu_{1}=0}^{k_{1}}...\sum\limits_{\mu_{m}=0}^{k_{m}}\prod_{j=1}^{m}2^{\mu_{j}\alpha_{j}\tau} Y_{[2^{\mu_{1}-1}],... [2^{\mu_{m}-1}]}^{\tau}(f)_{p, \tau}\Biggr\}^{1/\tau}
\end{multline}
for a function 
 $f\in L_{p, \tau}(\mathbb{T}^{m})$, in the case 
  $\beta = \min\{2, \tau\}=\tau$ and $1< p< \infty$.
By the property of the mixed modulus of smoothness (see Lemma 1.2) and by Theorem 2.1 \cite{2} we obtain
\begin{multline}\label{eq2.20}
I_{\overline{k}}:=\omega_{\overline\alpha}\bigl(S_{2^{\overline{k}^{e}, \infty}}(f-S_{\infty, 2^{\overline{k}^{\hat{e}}}}(f), \frac{1}{2^{k_{1}}},\ldots , \frac{1}{2^{k_{m}}}\bigr)_{p, \tau}\ll 
\prod_{j\in e}2^{-k_{j}\alpha_{j}}\|S_{2^{\overline{k}^{e}, \infty}}^{(\overline{\alpha}^{e})}(f-S_{\infty, 2^{\overline{k}^{\hat{e}}}}(f)\|_{p, \tau}
\\
\ll \prod_{j\in e}2^{-k_{j}\alpha_{j}}\biggl\|\sum\limits_{\overline{\mu}\in G_{\overline{k}}(e)}\delta_{\overline{\mu}}^{(\overline{\alpha}^{e})}(f) \biggr\|_{p, \tau} \ll \prod_{j\in e}2^{-k_{j}\alpha_{j}} \biggl\|\biggl(\sum\limits_{\overline{\mu}\in G_{\overline{k}}(e)}\prod_{j\in e}2^{\mu_{j}\alpha_{j}2} |\delta_{\overline{\mu}}(f)|^{2}\biggr)^{1/2} \biggr\|_{p, \tau}
\end{multline}
If $\beta = \min\{2, \tau\}=2$ and $2< p< \infty$, then by Lemma 1. 2 \cite{4} (see also \cite{5}) from \eqref{eq2.20} it follows that
\begin{multline}\label{eq2.21}
I_{\overline{k}}\ll \prod_{j\in e}2^{-k_{j}\alpha_{j}}\biggl(\sum\limits_{\substack{0\leqslant \mu_{j}\leqslant k_{j}, \\ j\in e}}\prod_{j\in e}2^{\mu_{j}\alpha_{j}2} \biggl\|\biggl(\sum\limits_{\substack{\mu_{j}>k_{j}, \\ j\in \hat{e}}} |\delta_{\overline{\mu}}(f)|^{2}\biggr)^{1/2}\biggr\|_{p, \tau}^{2}\biggr)^{1/2}
\\
\ll \prod_{j\in e}2^{-k_{j}\alpha_{j}}\biggl(\sum\limits_{\substack{0\leqslant \mu_{j}\leqslant k_{j},\\ j\in e}}\prod_{j\in e}2^{\mu_{j}\alpha_{j}2} \biggl\|\biggl(\sum\limits_{\substack{\nu_{j}>\mu_{j}, \\j\in e}}\sum\limits_{\substack{\mu_{j}>k_{j}, \\j\in \hat{e}}} |\delta_{\overline{\nu}^{e}, \overline{\mu}^{\hat{e}}}(f)|^{2}\biggr)^{1/2}\biggr\|_{p, \tau}\biggr)^{1/2}
\\
\prod_{j\in e}2^{-k_{j}\alpha_{j}}\biggl(\sum\limits_{\substack{0\leqslant \mu_{j}\leqslant k_{j},\\ j\in e}}\prod_{j\in e}2^{\mu_{j}\alpha_{j}2} \|f-U_{2^{\overline{\mu}^{e}}, 2^{\overline{k}^{\hat{e}}}}\|_{p, \tau}^{2} \biggr)^{1/2} 
\\
\ll \prod_{j\in e}2^{-k_{j}\alpha_{j}}\biggl(\sum\limits_{\substack{0\leqslant \mu_{j}\leqslant k_{j}, \\ j\in e}}\prod_{j\in e}2^{\mu_{j}\alpha_{j}2} Y_{2^{\overline{\mu}^{e}}, 2^{\overline{k}^{\hat{e}}}}(f)_{p, \tau}^{2} \biggr)^{1/2}.
\end{multline}
Since $\alpha_{j}>0$ and the best approximation by <<angle>> decreases with each index, then
\begin{equation}\label{eq2.22}
Y_{2^{\overline{\mu}^{e}}, 2^{\overline{k}^{\hat{e}}}}(f)_{p, \tau} \ll \prod_{j\in \hat{e}}2^{-k_{j}\alpha_{j}}\biggl(\sum\limits_{\substack{0\leqslant \mu_{j}\leqslant k_{j}, \\ j\in \hat{e}}}\prod_{j\in \hat{e}}2^{\mu_{j}\alpha_{j}2} Y_{2^{\overline{\mu}^{e}}, 2^{\overline{\mu}^{\hat{e}}}}(f)_{p, \tau}^{2} \biggr)^{1/2}.
\end{equation}
Therefore, from inequality \eqref{eq2.21} it follows that
\begin{equation}\label{eq2.23}
I_{\overline{k}}\ll \prod_{j=1}^{m}2^{-k_{j}\alpha_{j}}\biggl(\sum\limits_{\mu_{1}=0}^{k_{1}}...\sum\limits_{\mu_{m}=0}^{k_{m}}\prod_{j=1}^{m}2^{\mu_{j}\alpha_{j}2} Y_{2^{\mu_{1}},... 2^{\mu_{m}}}(f)_{p, \tau}^{2} \biggr)^{1/2}
\end{equation}
for a function  $f\in L_{p, \tau}(\mathbb{T}^{m})$, in the case 
  $\beta = \min\{2, \tau\}=2$ and $2< p< \infty$.

Let $\beta = \min\{2, \tau\}=\tau$ and $1< p< \infty$. Then according to Lemma 1.1 and Lemma 1.4 \cite{4} from \eqref{eq2.20} and \eqref{eq2.22} we obtain
\begin{multline}\label{eq2.24}
I_{\overline{k}}\ll \prod_{j\in e}2^{-k_{j}\alpha_{j}}\biggl(\sum\limits_{0\leqslant \mu_{j}\leqslant k_{j}, j\in e}\prod_{j\in e}2^{\mu_{j}\alpha_{j}\tau} \biggl\|\biggl(\sum\limits_{\mu_{j}>k_{j}, j\in \hat{e}} |\delta_{\overline{\mu}}(f)|^{2}\biggr)^{1/2}\biggr\|_{p, \tau}^{\tau}\biggr)^{1/\tau}
\\
\ll \prod_{j\in e}2^{-k_{j}\alpha_{j}}\biggl(\sum\limits_{0\leqslant \mu_{j}\leqslant k_{j}, j\in e}\prod_{j\in e}2^{\mu_{j}\alpha_{j}\tau} \biggl\|\biggl(\sum\limits_{\nu_{j}>\mu_{j}, j\in e}\sum\limits_{\mu_{j}>k_{j}, j\in \hat{e}} |\delta_{\overline{\nu}^{e}, \overline{\mu}^{\hat{e}}}(f)|^{2}\biggr)^{1/2}\biggr\|_{p, \tau}^{\tau}\biggr)^{1/\tau}
\\
\ll \prod_{j\in e}2^{-k_{j}\alpha_{j}}\biggl(\sum\limits_{0\leqslant \mu_{j}\leqslant k_{j}, j\in e}\prod_{j\in e}2^{\mu_{j}\alpha_{j}\tau} \|f-U_{2^{\overline{\mu}^{e}}, 2^{\overline{k}^{\hat{e}}}}\|_{p, \tau}^{\tau} \biggr)^{1/\tau} 
\\
\ll \prod_{j\in e}2^{-k_{j}\alpha_{j}}\biggl(\sum\limits_{0\leqslant \mu_{j}\leqslant k_{j}, j\in e}\prod_{j\in e}2^{\mu_{j}\alpha_{j}\tau} Y_{2^{\overline{\mu}^{e}}, 2^{\overline{k}^{\hat{e}}}}^{\tau}(f)_{p, \tau} \biggr)^{1/\tau} 
\\
\ll \prod_{j=1}^{m}2^{-k_{j}\alpha_{j}}\biggl(\sum\limits_{\mu_{1}=0}^{k_{1}}...\sum\limits_{\mu_{m}=0}^{k_{m}}\prod_{j=1}^{m}2^{\mu_{j}\alpha_{j}\tau} Y_{2^{\mu_{1}},... 2^{\mu_{m}}}^{\tau}(f)_{p, \tau} \biggr)^{1/\tau}.
\end{multline} 
Now from the inequalities \eqref{eq2.15}, \eqref{eq2.16}, \eqref{eq2.18}, \eqref{eq2.23} in the case $\beta = \min\{2, \tau\}=2$ and $2< p< \infty$ we obtain
\begin{equation}\label{eq2.25}
\omega_{\overline\alpha}\Bigl(f,\frac{\pi}{n_{1}},\ldots, \frac{\pi}{n_{m}}\Bigr)_{p, \tau}  \ll
\prod_{j=1}^{m}2^{-k_{j}\alpha_{j}}\biggl(\sum\limits_{\mu_{1}=0}^{k_{1}}...\sum\limits_{\mu_{m}=0}^{k_{m}}\prod_{j=1}^{m}2^{\mu_{j}\alpha_{j}2} Y_{[2^{\mu_{1}-1}],... [2^{\mu_{m}-1}]}^{2}(f)_{p, \tau} \biggr)^{1/2}
\end{equation}
for a function  $f\in L_{p, \tau}(\mathbb{T}^{m})$.

If $\beta = \min\{2, \tau\}=\tau$ and $1< p< \infty$, then from inequalities \eqref{eq2.15}, \eqref{eq2.16}, \eqref{eq2.19} and \eqref{eq2.24} it follows that
\begin{equation}\label{eq2.26} 
\omega_{\overline\alpha}\Bigl(f,\frac{\pi}{n_{1}},\ldots, \frac{\pi}{n_{m}}\Bigr)_{p, \tau}  \ll
\prod_{j=1}^{m}2^{-k_{j}\alpha_{j}}\biggl(\sum\limits_{\mu_{1}=0}^{k_{1}}...\sum\limits_{\mu_{m}=0}^{k_{m}}\prod_{j=1}^{m}2^{\mu_{j}\alpha_{j}\tau} Y_{[2^{\mu_{1}-1}],... [2^{\mu_{m}-1}]}^{\tau}(f)_{p, \tau} \biggr)^{1/\tau}.
\end{equation}
Now it is easy to verify that the statement of Theorem 2.2 follows from inequalities \eqref{eq2.25}, \eqref{eq2.26}.
\hfill $\Box$ 
 
\begin{rem}\label{rem2.4} In the case $\tau=p$, $m=2$ Theorem 2.1 and Theorem 2.2 are proved in \cite{17}.
\end{rem}

\setcounter{equation}{0}
\setcounter{lemma}{0}
\setcounter{theorem}{0}

\section{On lacunary Fourier series}\label{sec3}

 Now let us consider the relationship between the mixed modulus of smoothness and the best approximation by ''angle'' of a function with a lacunary Fourier series.

\begin{op}\label{def6}  Let  
$1 < p < +\infty,$ $1\leq \tau< \infty$. By $\Lambda_{p ,\tau}$ we denote the set of all functions $f\in {\mathring L}_{p, \tau} (\mathbb{T}^{m})$ that have a lacunary Fourier series 
\begin{equation*} 
f(\overline{x}) \sim \sum_{\overline{\nu} \in \mathbb{Z}_{+}^{m}}\lambda_{\overline{\nu}} \prod_{j=1}^{m}\cos 2^{\nu_{j}}x_{j}, \,\, \, \, \lambda_{\overline{\nu}}\in \mathbb{R}.
\end{equation*}
\end{op}
\begin{rem}\label{rem3.1} In the case $\tau = p$ the class $\Lambda_{p ,\tau}$ is defined in \cite{17}
\end{rem}

\begin{theorem}\label{th3.1}
Let $f \in \Lambda_{p ,\tau}$, $1 < p < +\infty,$ $1< \tau< \infty$.

1. If  $1 < p < +\infty,$ $1< \tau< \infty$, then 
\begin{equation*}
\|f\|_{p, \tau}\ll \Bigl(\sum_{\overline{\nu} \in \mathbb{Z}_{+}^{m}}|\lambda_{\overline{\nu}}|^{2}\Bigr)^{1/2}
\end{equation*}

2. If $1 < p \leq 2$ and $1< \tau \leq 2$ or  $2< p< +\infty$ and $1< \tau< \infty$, then
\begin{equation*}
\Bigl(\sum_{\overline{\nu} \in \mathbb{Z}_{+}^{m}}|\lambda_{\overline{\nu}}|^{2}\Bigr)^{1/2}\ll \|f\|_{p, \tau}. 
\end{equation*}

3. If $1 < p \leq 2$ and $2< \tau < \infty$, then
\begin{equation*}
\Bigl(\sum_{\overline{\nu} \in \mathbb{Z}_{+}^{m}}|\lambda_{\overline{\nu}}|^{\tau}\Bigr)^{1/\tau}\ll \|f\|_{p, \tau}. 
\end{equation*}
\end{theorem}
\begin{proof}
Since $\delta_{\overline{s}}(f, \overline{x}) = \lambda_{\overline{s}} \prod_{j=1}^{m}\cos 2^{s_{j}}x_{j}$, for the function $\Lambda_{p ,\tau}$, then by the Littlewood--Paley theorem in the Lorentz space (see \cite[Theorem 1.1]{25}) we have
\begin{equation*}
\|f\|_{p, \tau} \asymp \Bigl\|\Bigl(\sum_{\overline{s} \in \mathbb{Z}_{+}^{m}}\lambda_{\overline{s}}^{2}(\prod_{j=1}^{m}\cos 2^{s_{j}}x_{j})^{2} \Bigr)^{1/2}\Bigr\|_{p, \tau}\ll \Bigl(\sum_{\overline{s} \in \mathbb{Z}_{+}^{m}}\lambda_{\overline{s}}^{2}\Bigr)^{1/2},
\end{equation*}
for  $1 < p < +\infty,$ $1< \tau< \infty$. The first statement is proved.

If $2< p< +\infty$, then $L_{p, \tau}(\mathbb{T}^{m}) \subset L_{2}(\mathbb{T}^{m})$ and $\|f\|_{2}\ll \|f\|_{p, \tau}$, for $1< \tau< \infty$. Therefore, by Parseval's equality, we have
\begin{equation*}
\|f\|_{p, \tau}\gg \Bigl(\sum_{\overline{s} \in \mathbb{Z}_{+}^{m}}\lambda_{\overline{s}}^{2}\Bigr)^{1/2},
\end{equation*}
for  $2< p< +\infty$, $1< \tau< \infty$.

If $1 < p \leq 2$ and $1< \tau \leq 2$, then from \cite[Lemma 1.5 and Lemma 1.6]{4} it follows that
\begin{equation*}
\|f\|_{p, \tau} \gg \Bigl(\sum\limits_{\overline{s} \in \mathbb{Z}_{+}^{m}}\|\delta_{\overline{s}}(f)\|_{p, \tau}^{2}\Bigr)^{1/2} \gg \Bigl(\sum_{\overline{s} \in \mathbb{Z}_{+}^{m}}\lambda_{\overline{s}}^{2}\Bigr)^{1/2}.
\end{equation*}
If $1 < p \leq 2$ and $2< \tau < \infty$, then by \cite[Lemma 1.5 and Lemma 1.6]{4} we obtain
\begin{equation*}
\|f\|_{p, \tau} \gg \Bigl(\sum\limits_{\overline{s} \in \mathbb{Z}_{+}^{m}}\|\delta_{\overline{s}}(f)\|_{p, \tau}^{\tau}\Bigr)^{1/\tau} \gg \Bigl(\sum_{\overline{s} \in \mathbb{Z}_{+}^{m}}|\lambda_{\overline{s}}|^{\tau}\Bigr)^{1/\tau}.
\end{equation*}
\end{proof}

\begin{rem}\label{rem3.2}
In the case $\tau = p$, Theorem 3.1 coincides with Lemma 3. 9 \cite{17}.
\end{rem}
 
\begin{theorem}\label{th3.2} 
Let $f \in \Lambda_{p ,\tau}$, $1 < p < +\infty,$ $1< \tau< \infty$ and $\alpha_{j}> 0$ for $j=1,\ldots,m$.
Then 
\begin{multline}\label{eq3.1}
 \omega_{\overline\alpha}\Bigl(f,\frac{\pi}{n_{1}},\ldots, \frac{\pi}{n_{m}}\Bigr)_{p, \tau}  \ll \prod_{j=1}^{m}2^{-n_{j}\alpha_{j}}\biggl(\sum\limits_{k_{1}=0}^{n_{1}}...\sum\limits_{k_{m}=0}^{n_{m}}\prod_{j=1}^{m}2^{k_{j}\alpha_{j}2}|\lambda_{\overline{k}}|^{2} \biggr)^{1/2}
 \\
  + \sum_{\substack{e\subset e_{m},\\ e\neq \emptyset}} \prod_{j\in e}2^{-n_{j}\alpha_{j}}\biggl(\sum\limits_{\overline{k}\in G_{\overline{n}}(e)}|\lambda_{\overline{k}}|^{2}\prod_{j\in e}2^{-k_{j}\alpha_{j}2} \biggr)^{1/2} + \biggl(\sum\limits_{k_{1}=n_{1}+1}^{\infty}...\sum\limits_{k_{m}=n_{m}+1}^{\infty}|\lambda_{\overline{k}}|^{2} \biggr)^{1/2}.
\end{multline} 
If $1<p\leqslant 2$ and $1<\tau \leqslant 2$ or $2< p<\infty$ and $1<\tau <\infty$, then for 
the function $f \in \Lambda_{p ,\tau}$ the opposite inequality to \eqref{eq3.1} is true.
 \end{theorem}
\begin{proof}
According to Theorem 2.1 \cite{2} and Theorem 3.1 we have
 \begin{multline}\label{eq3.2}
\|S_{2^{n_{1}},...,2^{n_{m}}}^{(\overline{\alpha})}(f)\|_{p, \tau}
\ll \Biggl\|\Biggl(\sum\limits_{s_{1}=1}^{n_{1}}...\sum\limits_{s_{m}=1}^{n_{m}}\prod_{j=1}^{m}2^{s_{j}\alpha_{j}2}|\delta_{\overline{s}}(f)|^{2}\Biggr)^{1/2} \Biggr\|_{p, \tau} 
\\
\ll 
\Biggl(\sum\limits_{s_{1}=1}^{n_{1}}...\sum\limits_{s_{m}=1}^{n_{m}}\prod_{j=1}^{m}2^{s_{j}\alpha_{j}2}|\lambda_{\overline{s}}|^{2}\Biggr)^{1/2},
\end{multline} 
\begin{multline}\label{eq3.3} 
\|S_{2^{\overline{n}^{e}, \infty}}^{(\overline{\alpha}^{e})}(f-S_{\infty, 2^{\overline{n}^{\hat{e}}}}(f)\|_{p, \tau}
\ll \biggl\|\sum\limits_{\overline{s}\in G_{\overline{n}}(e)}\delta_{\overline{s}}^{(\overline{\alpha}^{e})}(f) \biggr\|_{p, \tau}
\\ 
 \ll  \biggl\|\biggl(\sum\limits_{\overline{s}\in G_{\overline{n}}(e)}\prod_{j\in e}2^{s_{j}\alpha_{j}2} |\delta_{\overline{s}}(f)|^{2}\biggr)^{1/2} \biggr\|_{p, \tau} \ll \biggl(\sum\limits_{\overline{s}\in G_{\overline{n}}(e)}\prod_{j\in e}2^{s_{j}\alpha_{j}2} |\lambda_{\overline{s}}|^{2}\biggr)^{1/2}, 
\end{multline}
\begin{equation}\label{eq3.4} 
\|f-U_{2^{n_{1}},...,2^{n_{m}}} (f)\|_{p, \tau} =\biggl\|\sum\limits_{s_{1}=n_{1}+1}^{\infty}...\sum\limits_{s_{m}=n_{m}+1}^{\infty} \delta_{\overline{s}}(f) \biggr\|_{p, \tau} \ll \biggl(\sum\limits_{s_{1}=n_{1}+1}^{\infty}...\sum\limits_{s_{m}=n_{m}+1}^{\infty}|\lambda_{\overline{s}}|^{2}  \biggr)^{1/2}
\end{equation}
for a function   $f \in \Lambda_{p ,\tau}$, $1 < p < +\infty,$ $1< \tau< \infty$.
 Now, according to Theorem 2.1, inequalities \eqref{eq3.2}--\eqref{eq3.4} imply inequality \eqref{eq3.1}.
 
 In the case 
$1<p\leqslant 2$ and $1<\tau \leqslant 2$ or $2< p<\infty$ and $1<\tau <\infty$, according to the second assertion of Theorem 3.1 for a function $f \in \Lambda_{p ,\tau}$ the following inequalities hold 
 \begin{multline}\label{eq3.5}
\|S_{2^{n_{1}},...,2^{n_{m}}}^{(\overline{\alpha})}(f)\|_{p, \tau}
\gg \Biggl\|\Biggl(\sum\limits_{s_{1}=1}^{n_{1}}...\sum\limits_{s_{m}=1}^{n_{m}}\prod_{j=1}^{m}2^{s_{j}\alpha_{j}2}|\delta_{\overline{s}}(f)|^{2}\Biggr)^{1/2} \Biggr\|_{p, \tau} 
\\
\gg 
\Biggl(\sum\limits_{s_{1}=1}^{n_{1}}...\sum\limits_{s_{m}=1}^{n_{m}}\prod_{j=1}^{m}2^{s_{j}\alpha_{j}2}|\lambda_{\overline{s}}|^{2}\Biggr)^{1/2},
\end{multline} 
\begin{multline}\label{eq3.6} 
\|S_{2^{\overline{n}^{e}, \infty}}^{(\overline{\alpha}^{e})}(f-S_{\infty, 2^{\overline{n}^{\hat{e}}}}(f)\|_{p, \tau}
\gg \biggl\|\sum\limits_{\overline{s}\in G_{\overline{n}}(e)}\delta_{\overline{s}}^{(\overline{\alpha}^{e})}(f) \biggr\|_{p, \tau}
\\ 
 \gg  \biggl\|\biggl(\sum\limits_{\overline{s}\in G_{\overline{n}}(e)}\prod_{j\in e}2^{s_{j}\alpha_{j}2} |\delta_{\overline{s}}(f)|^{2}\biggr)^{1/2} \biggr\|_{p, \tau} \gg \biggl(\sum\limits_{\overline{s}\in G_{\overline{n}}(e)}\prod_{j\in e}2^{s_{j}\alpha_{j}2} |\lambda_{\overline{s}}|^{2}\biggr)^{1/2}, 
\end{multline}
\begin{equation}\label{eq3.7} 
\|f-U_{2^{n_{1}},...,2^{n_{m}}} (f)\|_{p, \tau} =\biggl\|\sum\limits_{s_{1}=n_{1}+1}^{\infty}...\sum\limits_{s_{m}=n_{m}+1}^{\infty} \delta_{\overline{s}}(f) \biggr\|_{p, \tau} \gg \biggl(\sum\limits_{s_{1}=n_{1}+1}^{\infty}...\sum\limits_{s_{m}=n_{m}+1}^{\infty}|\lambda_{\overline{s}}|^{2}  \biggr)^{1/2}.
\end{equation}
Therefore, by virtue of Theorem 2.1, the desired statement follows from inequalities \eqref{eq3.5}--\eqref{eq3.7}.
\end{proof}
 
\begin{theorem}\label{th3.3} 
Let  $f \in \Lambda_{p, \tau}$, $1<p\leqslant 2$ and $1<\tau \leqslant 2$ or $2< p<\infty$ and $1<\tau <\infty$, $\alpha_{j}> 0$ for $j=1,\ldots,m$. Then
 \begin{equation*}  
\omega_{\overline\alpha}\Bigl(f,\frac{1}{n_{1}},\ldots, \frac{1}{n_{m}}\Bigr)_{p, \tau}  \asymp
\prod_{j=1}^{m}n^{-\alpha_{j}}\biggl(\sum\limits_{\nu_{1}=1}^{n_{1}+1}...\sum\limits_{\nu_{m}=1}^{n_{m}+1}\prod_{j=1}^{m}\nu_{j}^{2\alpha_{j}-1} Y_{\nu_{1}-1,... \nu_{m}-1}^{2}(f)_{p, \tau} \biggr)^{1/2}.
\end{equation*}
 \end{theorem}
\begin{proof}
We introduce the notation  
 \begin{equation*}
 I_{\overline{n}}(f)_{p, \tau}:=\prod_{j=1}^{m}n^{-\alpha_{j}}\biggl(\sum\limits_{\nu_{1}=1}^{n_{1}+1}...\sum\limits_{\nu_{m}=1}^{n_{m}+1}\prod_{j=1}^{m}\nu_{j}^{2\alpha_{j}-1} Y_{\nu_{1}-1,... \nu_{m}-1}^{2}(f)_{p, \tau} \biggr)^{1/2}.
 \end{equation*}
For $n_{j}\in \mathbb{N}$, choose a non-negative integer $l_{j}$ such that $2^{l_{j}}\leqslant n_{j}< 2^{l_{j}+1}$, $j=1,\ldots,m$. Then, by the property of monotone decrease of the sequence $\{Y_{\nu_{1},... \nu_{m}}(f)_{p, \tau}\}$ for each index $\nu_{j}$ we have
  \begin{equation}\label{eq3.8}
 I_{\overline{n}}(f)_{p, \tau} \asymp \prod_{j=1}^{m}2^{-l_{j}\alpha_{j}}\biggl(\sum\limits_{k_{1}=1}^{l_{1}+1}...\sum\limits_{k_{m}=1}^{l_{m}+1}\prod_{j=1}^{m}2^{k_{j}2\alpha_{j}}Y_{2^{k_{1}-1},... ,2^{k_{m}-1}}^{2}(f)_{p, \tau} \biggr)^{1/2}.
\end{equation} 
Further, using Lemma 1.5 and the first statement of Theorem 3.1 from \eqref{eq3.8} we obtain 
   \begin{equation}\label{eq3.9}
 I_{\overline{n}}^{2}(f)_{p, \tau} \ll \prod_{j=1}^{m}2^{-l_{j}\alpha_{j}2}\sum\limits_{k_{1}=1}^{l_{1}+1}...\sum\limits_{k_{m}=1}^{l_{m}+1}\prod_{j=1}^{m}2^{k_{j}2\alpha_{j}}\sum\limits_{\nu_{1}=k_{1}}^{\infty}...\sum\limits_{\nu_{m}=k_{m}}^{\infty}|\lambda_{\overline{\nu}}|^{2} 
\end{equation} 
for a function 
 $f \in \Lambda_{p ,\tau}$.
Now, on the right-hand side of inequality \eqref{eq3.9}, changing the order of summation and taking into account that $\alpha_{j}>0$ for $j=1,\ldots,m$ we obtain
    \begin{multline}\label{eq3.10}
 I_{\overline{n}}^{2}(f)_{p, \tau} \ll \prod_{j=1}^{m}2^{-l_{j}\alpha_{j}2}\sum\limits_{\nu_{1}=1}^{l_{1}}...\sum\limits_{\nu_{m}=1}^{l_{m}}\prod_{j=1}^{m}2^{\nu_{j}\alpha_{j}2}|\lambda_{\overline{\nu}}|^{2}
 \\
  + \sum_{\substack{e\subset e_{m},\\ e\neq \emptyset}} \prod_{j\in e}2^{-l_{j}\alpha_{j}2}\sum\limits_{\overline{\nu}\in G_{\overline{l}}(e)}|\lambda_{\overline{\nu}}|^{2}\prod_{j\in e}2^{\nu_{j}\alpha_{j}2}  + \sum\limits_{k_{1}=l_{1}+1}^{\infty}...\sum\limits_{k_{m}=l_{m}+1}^{\infty}|\lambda_{\overline{k}}|^{2}
\end{multline} 
for a function $f \in \Lambda_{p ,\tau}$.  Therefore, according to the second statement of Theorem 3.2 and inequality \eqref{eq3.10}, the inequality holds
 \begin{equation*}
I_{\overline{n}}(f)_{p, \tau} \ll \omega_{\overline\alpha}\Bigl(f,\frac{1}{n_{1}},\ldots, \frac{1}{n_{m}}\Bigr)_{p, \tau}
\end{equation*}
for a function $f \in \Lambda_{p ,\tau}$ in the case 
  $1<p\leqslant 2$ and $1<\tau \leqslant 2$ or $2< p<\infty$ and $1<\tau <\infty$. 
  
Taking into account the second assertion of Theorem 3.1, similarly to the proof of \eqref{eq3.10}, we can verify that
    \begin{multline}\label{eq3.11}
 I_{\overline{n}}^{2}(f)_{p, \tau} \gg \prod_{j=1}^{m}2^{-l_{j}\alpha_{j}2}\sum\limits_{\nu_{1}=1}^{l_{1}}...\sum\limits_{\nu_{m}=1}^{l_{m}}\prod_{j=1}^{m}2^{\nu_{j}\alpha_{j}2}|\lambda_{\overline{\nu}}|^{2}
 \\
  + \sum_{\substack{e\subset e_{m},\\ e\neq \emptyset}} \prod_{j\in e}2^{-l_{j}\alpha_{j}2}\sum\limits_{\overline{\nu}\in G_{\overline{l}}(e)}|\lambda_{\overline{\nu}}|^{2}\prod_{j\in e}2^{\nu_{j}\alpha_{j}2}  + \sum\limits_{k_{1}=l_{1}+1}^{\infty}...\sum\limits_{k_{m}=l_{m}+1}^{\infty}|\lambda_{\overline{k}}|^{2}
\end{multline}
  for a function    $f \in \Lambda_{p ,\tau}$ in the case  
   $1<p\leqslant 2$ and $1<\tau \leqslant 2$ or $2< p<\infty$ and $1<\tau <\infty$.
Now, using the first assertion of Theorem 3.2 and inequality \eqref{eq3.11}, we obtain
$\omega_{\overline\alpha}\Bigl(f,\frac{1}{n_{1}},\ldots, \frac{1}{n_{m}}\Bigr)_{p, \tau}\ll I_{\overline{n}}(f)_{p, \tau}$ 
for a function $f \in \Lambda_{p ,\tau}$ in the case $1<p\leqslant 2$ and $1<\tau \leqslant 2$ or $2< p<\infty$ and $1<\tau <\infty$. 
 \end{proof}
\begin{rem}\label{rem3.3}
In the case $\tau = p$, Theorem 3.2 and Theorem 3.3 were previously proved in \cite[Theorem 7.3 and Theorem 9.3]{17}.
\end{rem}  
  

\begin{flushleft}
Kazakhstan Branch , Lomonosov Moscow University\\
St. Kazhymukan, 11 \\
010010, Astana, Kazakhstan 
 E-mail: akishev\_g@mail.ru
\end{flushleft}
\end{document}